\documentclass[12pt,reqno]{article}

\usepackage[usenames]{color}
\usepackage{amssymb}
\usepackage{graphicx}
\usepackage{amscd}

\usepackage[colorlinks=true,
linkcolor=webgreen,
filecolor=webbrown,
citecolor=webgreen]{hyperref}

\definecolor{webgreen}{rgb}{0,.5,0}
\definecolor{webbrown}{rgb}{.6,0,0}

\usepackage{color}
\usepackage{fullpage}
\usepackage{float}

\usepackage{graphics,amsmath,amssymb}
\usepackage{amsthm}
\usepackage{amsfonts}
\usepackage{latexsym}
\usepackage{epsf}

\usepackage{verbatim}
\usepackage{enumerate}

\usepackage{tikz}

\usepackage[blocks]{authblk}

\usepackage{algpseudocode}

\theoremstyle{plain}
\newtheorem{theorem}{Theorem}

\newtheorem{lemma}[theorem]{Lemma}
\newtheorem{proposition}[theorem]{Proposition}

\theoremstyle{definition}

\theoremstyle{remark}

\begin{document}

\title{Quad Squares}
\author{Nikhil Byrapuram}
\author{Hwiseo (Irene) Choi}
\author{Adam Ge}
\author{Selena Ge}
\author{Sylvia Zia Lee}
\author{Evin Liang}
\author{Rajarshi Mandal}
\author{Aika Oki}
\author{Daniel Wu}
\author{Michael Yang}
\affil{PRIMES STEP}
\author{Tanya Khovanova}
\affil{MIT}

\maketitle

\begin{abstract}
We study 4-by-4 squares formed by cards from the EvenQuads deck. EvenQuads is a card game with 64 cards where cards have 3 attributes with 4 values in each attribute. A quad is four cards with all attributes the same, all different, or half and half. We define Latin quad squares as squares where the cards in each row and column have different values for each attribute. We define semimagic quad squares as squares where each row and column form a quad. For magic quad squares, we add a requirement that the diagonals have to form a quad. We also define strongly magic quad squares.

We analyze types of semimagic and strongly magic quad squares. We also calculate the number of semimagic, magic, and strongly magic quad squares for quad decks of any size.

These squares can be described in terms of integers. Four integers form a quad when their bitwise XOR is zero.
\end{abstract}

\section{Introduction}

It all started with the game of SET: a trendy card game with a lot of mathematics hidden in it \cite{MGGG}. The game of SET is played with a deck of 81 cards. Each card is characterized by 4 attributes:
\begin{itemize}
\item Number: 1, 2, or 3 symbols.
\item Color: Green, red, or purple.
\item Shading: Empty, striped, or solid.
\item Shape: Oval, diamond, or squiggle.
\end{itemize}

A \textit{set} is formed by three cards that are either all the same or all different in each attribute. The players try to find sets among given cards, and the fastest player wins.

A \textit{magic SET square} is a 3-by-3 grid of SET cards such that each row, column, and diagonal form a set. The magic set squares were classified in \cite{STEPMSS}.

In this paper, we are interested in a different game that is a generalization of the SET game introduced by Rose and Perreira \cite{Rose}. It was initially called SuperSET, but now it is called EvenQuads.

The EvenQuads deck consists of $64 = 4^3$ cards with different objects. The cards have 3 attributes with 4 values each:
\begin{itemize}
\item Number: 1, 2, 3, or 4 symbols.
\item Color: Red, green, yellow, or blue.
\item Shape: Square, icosahedron, circle, or spiral.
\end{itemize}

A \textit{quad} consists of four cards so that for each attribute, the values of the cards must be one of three cases: all the same, all different, or half and half. In our paper \cite{LAGames}, we have also presented the study of several properties and their connections to linear algebra of the games SET and EvenQuads, as well as two other famous games Socks and Spot It!.

We can translate this game from cards to numbers. Consider a card. We can assign a binary string of length 2 representing one of the four values for each attribute. Thus we can view a card as a binary string of length 6. Four cards form a quad if and only if the bitwise XOR of the corresponding four strings is a zero string. Bitwise XOR is also called a parity sum and can be described as addition modulo 2 without carry.

We can convert binary to decimal and view each card as a number. For example, quadruples $\{0,1,2,3\}$ and $\{0,4,8,12\}$ are quads. The official quad deck has 64 cards, and we correspond the cards to integers from 0 to 63. But we can extend this to a deck whose size is any power of 2. We are interested in decks with at least 16 cards, as we will build different 4-by-4 squares out of them.

This paper focuses on analogs of magic SET squares for EvenQuads cards. The magic SET squares are 3-by-3 squares of SET cards so that every row, column, and diagonal forms a set. Analogously, we study squares formed by 16 EvenQuads cards laid out as a 4-by-4 grid. We start with Latin squares, where all rows and columns must be quads of type D (all different). We also study magic and semimagic squares, where different lines are required to be quads. We also introduce a notion of a strongly magic square, where 4 cards with coordinates forming a quad have to form a quad. Latin, magic, and semimagic quad squares are strongly magic squares.

In Section~\ref{sec:differentsquares}, we define and discuss the above-mentioned squares. We use symmetries to reduce semimagic, magic, and strongly magic quad squares to squares with the first row equal to 0, 1, 2, 3, and the first column equal to 0, 4, 8, 12, which we call squares of type C. We prove that the number of semimagic, magic, and strongly magic quad squares using the EvenQuads-$2^n$ deck is $2^n(2^n - 1)(2^n - 2)(2^n - 4)(2^n - 8)$ times the number of corresponding squares of type C.

In Section~\ref{sec:stronglymagicsquares}, we start with enumerating strongly magic quad squares for any deck. For the standard deck, the number of such squares is 839946240. We discuss types of squares and classify all possible strongly magic quad squares for the EvenQuads-16 deck for one attribute. We present 7 cases up to symmetries.

In Section~\ref{sec:numbersmsquares}, we enumerate semimagic quad squares for any sized decks. In particular, the number of semimagic quad squares for the standard deck is 3379298591047680. In Section~\ref{sec:numbermsquares}, we enumerate magic quad squares for any sized decks. In particular, the number of magic quad squares for the standard deck is 58912149381120.

\section{Different Quad Squares}
\label{sec:differentsquares}

\subsection{Quads}

Recall that the EvenQuads deck consists of $64 = 4^3$ cards with different objects. The cards have 3 attributes with 4 values each:
\begin{itemize}
\item Number: 1, 2, 3, or 4 symbols.
\item Color: Red, green, yellow, or blue.
\item Shape: Square, icosahedron, circle, or spiral.
\end{itemize}

A \textit{quad} consists of four cards so that for each attribute, the values of the cards must be one of three cases: all the same, all different, or half and half. We can assume that each attribute takes values in the set $\{0,1,2,3\}$. Then four cards form a quad if and only if the bitwise XOR of the values in each attribute is zero. This is equivalent to saying that each attribute takes values in $\mathbb{Z}_2^2$. Thus, we can view our cards as vectors in $\mathbb{Z}_2^6$. For generalizations, we can consider EvenQuads deck of sizes $2^n$. They correspond to vectors in $\mathbb{Z}_2^n$.

Four vectors $\vec{a}$, $\vec{b}$, $\vec{c}$, and $\vec{d}$ form a quad if and only if (see \cite{CragerEtAl})
\[\vec{a} + \vec{b} + \vec{c} + \vec{d} = \vec{0}.\]
Consider four vectors $\vec{a}$, $\vec{b}$, $\vec{c}$, and $\vec{d}$ forming a quad. By a translation, we can assume that $\vec{a}$ is the origin. Then, from the above equation we get $\vec{d} = - \vec{b} - \vec{c} = \vec{b} + \vec{c}$. This means that vector $\vec{d}$ is in the same plane as the origin and vectors $\vec{b}$ and $\vec{c}$. Thus, four cards form a quad if and only if their endpoints belong to the same plane.

In the rest of the paper, we often use numbers from 0 to $2^n-1$ inclusive to label the cards in the EvenQuads-$2^n$ deck. Four numbers form a quad, if and only if their bitwise XOR is 0.

\subsection{Latin quad squares}

A \textit{Latin square} is an $n$-by-$n$ grid filled with $n$ different symbols, each occurring exactly once in each row and exactly once in each column. There are 576 Latin squares of order 4; see sequence A002860 in the OEIS \cite{OEIS}. A Latin square is called \textit{reduced} if the first column and the first row are in order. We can achieve a reduced Latin square by permuting rows and columns. There are 4 reduced Latin squares of order 4; see sequence A000315.

A natural way to define a \textit{Latin quad square} is as a 4-by-4 grid of EvenQuads cards, such that considering every attribute separately, the grid forms a Latin square. In other words, each Latin quad square is a 4-by-4 grid of EvenQuads cards, where each row and column form a quad of type $D$. As we will see in the next section, we can view a Latin quad square as a Latin semimagic quad square.

One might ask how many different Latin quad squares are possible. It is known that there are 576 different Latin squares of size 4. That means there should be $576^3$ Latin quad squares. Right?

Wrong. Table~\ref{table:impossiblels} shows three Latin squares. What will happen if we use each square as the value for different attributes? The resulting Latin quad square cannot exist because cards 000 and 111 are repeated.

\begin{table}[ht!]
\begin{center}
\begin{tabular} {|c|c|c|c|}
\hline
0 & 1 & 2 & 3 \\ \hline
1 & 0 & 3 & 2 \\ \hline
2 & 3 & 0 & 1 \\ \hline
3 & 2 & 1 & 0 \\ \hline
\end{tabular}
\quad
\begin{tabular} {|c|c|c|c|}
\hline
0 & 1 & 2 & 3 \\ \hline
1 & 0 & 3 & 2 \\ \hline
2 & 3 & 1 & 0 \\ \hline
3 & 2 & 0 & 1 \\ \hline
\end{tabular}
\quad
\begin{tabular} {|c|c|c|c|}
\hline
0 & 1 & 2 & 3 \\ \hline
1 & 0 & 3 & 2 \\ \hline
3 & 2 & 0 & 1 \\ \hline
2 & 3 & 1 & 0 \\ \hline
\end{tabular} \indent
\end{center}
\caption{A combination that does not match a Latin quad square.}
\label{table:impossiblels}
\end{table}

However, Table~\ref{table:possiblels} provides a real example of a Latin quad square, where each particular Latin square corresponds to one attribute.
\begin{table}[ht!]
\begin{center}
\begin{tabular}{|c|c|c|c|}
\hline
0 & 1 & 2 & 3 \\ \hline
1 & 3 & 0 & 2 \\ \hline
2 & 0 & 3 & 1 \\ \hline
3 & 2 & 1 & 0 \\ \hline
\end{tabular} \indent
\quad
\begin{tabular} {|c|c|c|c|}
\hline
0 & 1 & 3 & 2 \\ \hline
2 & 3 & 1 & 0 \\ \hline
3 & 2 & 0 & 1 \\ \hline
1 & 0 & 2 & 3 \\ \hline
\end{tabular} \indent
\quad
\begin{tabular} {|c|c|c|c|}
\hline
0 & 3 & 1 & 2 \\ \hline
3 & 2 & 0 & 1 \\ \hline
2 & 1 & 3 & 0 \\ \hline
1 & 0 & 2 & 3 \\ \hline
\end{tabular}
\end{center}
\caption{A combination that generates a Latin quad square.}
\label{table:possiblels}
\end{table}

In any case, we can say that there are no more than  $576^3$ Latin quad squares.

Suppose we built a Latin quad square using an EvenQuads-16 deck. Then each attribute corresponds to a Latin square. Moreover, the fact that all the cards are different means that the two Latin squares are mutually orthogonal \cite{DK}. It is known that there are 6912 ordered pairs of mutually orthogonal Latin squares of order 4, as can be seen in sequence A072377. We can use symmetries and assume that the first square has the first column and row in order. This reduces that number of pairs by $24 \cdot 6 = 144$. We can also assume, using relabeling of the numbers, that the second square has the first row in order. Thus, the number of ordered pairs of mutually orthogonal Latin squares up to symmetries is 2.

In a way, a Latin quad square is a generalization of mutually orthogonal Latin squares. Indeed, as we see in Table~\ref{table:possiblels}, to get a Latin quad square, we do not need two squares corresponding to two different attributes to be orthogonal.

It is also well-known that the maximum number of mutually orthogonal Latin squares of order 4 is 3. We can see an example of such squares in Table~\ref{table:mols}. We can build a Latin quad square using the original deck and each square in the table as a value for one attribute. Moreover, we can pick 16 cards from the standard deck, fixing the value of one attribute. We can use any two of these squares to build a Latin square out of these 16 cards.
\begin{table}[ht!]
\begin{center}
\begin{tabular}{|c|c|c|c|}
\hline
0 & 1 & 2 & 3 \\ \hline
1 & 0 & 3 & 2 \\ \hline
2 & 3 & 0 & 1 \\ \hline
3 & 2 & 1 & 0 \\ \hline
\end{tabular} \indent
\quad
\begin{tabular} {|c|c|c|c|}
\hline
0 & 1 & 2 & 3 \\ \hline
2 & 3 & 0 & 1 \\ \hline
3 & 2 & 1 & 0 \\ \hline
1 & 0 & 3 & 2 \\ \hline
\end{tabular} \indent
\quad
\begin{tabular} {|c|c|c|c|}
\hline
0 & 1 & 2 & 3 \\ \hline
3 & 2 & 1 & 0 \\ \hline
1 & 0 & 3 & 2 \\ \hline
2 & 3 & 0 & 1 \\ \hline
\end{tabular}
\end{center}
\caption{Three mutually orthogonal Latins squares.}
\label{table:mols}
\end{table}

\subsection{Magic and semimagic quad squares}

A square array of numbers is called a \textit{magic square} if the sums of the numbers in each row, column, and diagonal are the same. A square array of numbers is called a \textit{semimagic square} if the sums of the numbers in each row and each column are the same. In other words, the condition on the diagonals is dropped.

All rotations and reflections of a square keep the property of the square to be magic or semimagic. Semimagic squares have some additional symmetries. For example, we can shuffle rows and columns. Thus, we can assume that the first row and the first column are in non-decreasing order. Moreover, if the first row has the same value for two columns, we continue shuffling the columns and assume that the second row for a given value in the first row is in non-decreasing order. 

The notion of a magic square is known for the game of SET. For such a game, it is defined as a 3-by-3 square made out of SET cards where each row, column, and diagonal forms a SET.

There is a natural way to generalize it to EvenQuads. We define a \textit{magic quad square} as a 4-by-4 square of EvenQuads cards such that each row, column, and diagonal forms a quad. Equivalently, for a chosen attribute, the values in each row, column, and diagonal are all different, all the same, or half and half. We define a \textit{semimagic quad square} as a 4-by-4 square of cards such that each row and column forms a quad. 

Here is an example of a semimagic but not magic quad square restricted to one attribute:
\begin{center}
\begin{tabular}{|c|c|c|c|}
\hline
$0$ & $1$ & $2$ & $3$ \\ \hline
$2$ & $3$ & $0$ & $1$ \\ \hline
$3$ & $0$ & $1$ & $2$ \\ \hline
$1$ & $2$ & $3$ & $0$ \\ \hline
\end{tabular} .
\end{center}

We can describe the magic and semimagic quad squares not only in terms of cards but also in terms of numbers. A magic quad square is a 4-by-4 grid of numbers such that each bitwise XOR of rows, columns, and diagonals is zero. We can define a semimagic quad square in terms of numbers similarly.

\subsection{Strongly magic quad squares}

In the game of SET, the magic squares have the following property. By definition, each magic SET square contains eight sets: in each row, column, and diagonal. It always contains four more sets. Any three cards in a magic square, such that the values for each coordinate are either all the same or all different, form a set. Equivalently, if the coordinates of three cards form a set, then the cards themselves form a set. We use this idea in the next definition.

We define a \textit{strongly magic quad square} to be a square such that for any four cards, if their $x$-coordinates are all the same, half and half, or all different, and their $y$-coordinates are all the same, half and half, or all different, then the four cards form a quad. Equivalently, if the coordinates of four cards form a quad, then the cards form a quad.

In terms of SET, a strongly magic SET square is the same as a magic SET square. In the game of quads, this property breaks. A strongly magic quad square is a magic quad square and, therefore, a semimagic quad square. However, a semimagic or magic square does not have to be a strongly magic square. Here is an example of a magic square that is not a strongly magic square, where the bolded values do not form a quad as needed: 
\begin{center}
\begin{tabular}{|c|c|c|c|}
\hline
     \textbf{000}& 010 & 020 & 030 \\
\hline 
     003 & 011 & 023 & \textbf{031} \\
\hline
     002 & \textbf{012} & 022 & 032 \\
\hline 
     001 & 013 & \textbf{021} & 033 \\
\hline 
\end{tabular} .
\end{center}

We can view the 16 pairs of coordinates for locations of different cards in a strongly magic quad square as a Quad-16 deck. Any three cards in a Quad-16 deck can be completed to a quad. Thus, for any three cards in a strongly magic quad square, there is a card that completes the three to a quad. We just need to take the coordinates of the three given cards, and for the fourth card, pick a card whose coordinates complete the three given sets of coordinates into a quad.

Each row of a strongly magic quad square is a quad and, therefore, a plane in our space. Different rows correspond to parallel planes. We show this by the following argument. Let the cards in the first row from left to right be $a_1$, $a_2$, $a_3$, and $a_4$, the cards in the second row from left to right be $b_1$, $b_2$, $b_3$, and $b_4$, etc. We assume that $b_1=a_1+d$. Let $i$ be an index 2, 3, or 4. Then, the cards $a_1$, $a_i$, $b_1$, and $b_i$ form a quad by definition of a strongly magic quad square. Thus, $a_1+a_i+b_1+b_i = 0$. Replacing $b_1$ with $a_1 + d$, we get $a_1+a_i+a_1+d+b_i =  a_i + d +b_i = 0$. Thus, $b_i = a_i + d$. Thus, the second row is the translation of the first row by vector $d$. Similar statements are true for other rows. Therefore, the rows form parallel planes. An analogous statement is true for columns.

We can define strongly magic quad squares in terms of numbers, rather than cards, in the following way. A 4-by-4 grid of numbers forms a strongly magic quad square where four numbers such that four coordinates in each direction bitwise XOR to zero have to bitwise XOR to zero themselves.

\subsection{Calculating the number of magic, semimagic, and strongly magic quad squares}

Consider a special example of a strongly magic quad square with cards from the EvenQuads-16 deck:
\begin{center}
\begin{tabular}{|c|c|c|c|}
\hline
0 & 1 & 2 & 3 \\ \hline
4 & 5 & 6 & 7 \\ \hline
8 & 9 & 10 & 11 \\ \hline
12 & 13 & 14 & 15 \\ \hline
\end{tabular} .
\end{center}

We call a quad square \textit{type C} if the first row is 0, 1, 2, 3 and the first column is 0, 4, 8, 12. 

Squares of type C play an important role in calculating the total number of quad squares. As mentioned earlier, we can view our EvenQuads-$2^n$ deck as a vector space $\mathbb{Z}^n$. As quads correspond to planes, we can view affine transformations of $\mathbb{Z_2}^n$ as a symmetry group of the deck. By using affine transformations, we can reduce any quads square to type C.

\begin{lemma}
\label{lemma:symmetrymultiplier}
The number of semimagic, magic, and strongly magic quad squares using the EvenQuads-$2^n$ deck is $2^n(2^n-1)(2^n-2)(2^n-4)(2^n-8)$ times the number of semimagic, magic, and strongly magic quad squares of type C in the same deck.
\end{lemma}

\begin{proof}
Consider the first row of a quad square consisting of cards $a$, $b$, $c$, and $a+b+c$. Since all cards are equivalent, we may suppose $a=0$, count all the quad squares with $a=0$, and then multiply by $2^n$ to get the total number. All remaining cards are equivalent; hence, we can let $b=1$ and add a factor of $2^n-1$ for future multiplication. All remaining cards are still equivalent; hence we can let $c=2$, adding a factor of $2^n-2$. Then the last card in the first row is forced, and we proceed with the first column:
\begin{center}
\begin{tabular}{|c|c|c|c|}
\hline
$a$&$b$&$c$&$a+b+c$\\\hline
$d$& ? & ? & ? \\\hline
$e$& ? & ? & ? \\\hline
$a+d+e$& ? & ? & ? \\
\hline
\end{tabular} .
\end{center}

Now $d$ cannot be 0, 1, 2, or 3, and all numbers other than 0, 1, 2, and 3 are equivalent. Hence we can let $d=4$, adding a factor of $2^n-4$ for future multiplication. Now $e$ cannot be any number from 0 to 7 since either $e$ will be from 0 to 4 or $a+d+e$ will be from 0 to 4, leading to a repeated value. All numbers not 0 through 7 are equivalent; hence we may let $d=8$ and add a factor of $2^n-8$. Then the last element in the first column has to be $a+d+e=12$.

At this point, our generic quad square looks as follows:
\begin{center}
\begin{tabular}{|c|c|c|c|}
\hline
0&1&2&3\\\hline
4& ? & ? & ? \\\hline
8& ? & ? & ?\\\hline
12& ? & ? & ? \\ \hline
\end{tabular} '
\end{center}
and the total number of quad squares is $2^n(2^n-1)(2^n-2)(2^n-4)(2^n-8)$ times the number of quad squares of this form.
\end{proof}

\section{Types of Semimagic Quad Squares}
\label{sec:typesofsmsquares}

We classify all possible types of semimagic quad squares according to the types of rows and columns. As we can shuffle rows and columns, we can arrange types of rows in lexicographic order. There are 15 possible sets of 4 types: DDDD, DDDH, DDDS, DDHH, DDHS, DDSS, DHHH, DHHS, DHSS, DSSS, HHHH, HHHS, HHSS, HSSS, and SSSS. Not all of these types are realizable. For example, if the first three rows are S rows, the three first values in each column are the same. They uniquely define the fourth value. Thus, the fourth row has to be S too. Similarly, if the first two rows are S, the two values in the third and fourth rows of each column have to be the same. Thus, the other two rows have to be the same type.

We can look at how values are distributed within types. For example, DDDD has four of each value; we can write this as 4,4,4,4. Type DDDS has three instances of three values and seven instances of the other value, which we can write as 3,3,3,7. Types H and S contribute an even number of each value. Type D contributes an odd number of each value. Thus all numbers in the distribution are the same parity, depending on whether the number of D rows/columns is even or odd. Thus, we can call each distribution either \textit{even} or \textit{odd}.

Table~\ref{table:dist} shows possible value distributions for each type of quad combination in a semimagic quad square. A value distribution or a type is not included if there is no semimagic quad square of this type. The data in Table~\ref{table:dist} is divided into two tables corresponding to even and odd distributions. We mark the possible distributions as Yes.

\begin{table}[ht!]
\begin{center}
\begin{tabular}{|c|c|c|c|c|c|c|c|c|}
\hline
Even     & DDDD      & DDHH    & DDHS    & DDSS     & HHHH    & HHHS     & HHSS        & SSSS                        \\ \hline
0,0,0,16 &           &         &         &          &         &          &             & Yes \\ \hline
0,0,4,12 &           &         &         &          &         &          & Yes         &                             \\ \hline
0,0,6,10 &           &         &         &          &         & Yes      &             &  \\ \hline
0,0,8,8  &           &         &         &          & Yes     &          & Yes         &  Yes                           \\ \hline
0,4,4,8  &           &         &         &          & Yes     & Yes      & Yes         &                             \\ \hline
0,4,6,6  &           &         &         &          & Yes     &          &             &                             \\ \hline
2,2,2,10 &           &         &         & Yes      &         & Yes      &             &                             \\ \hline
2,2,4,8  &           &         & Yes     &          & Yes     &          &             &                             \\ \hline
2,2,6,6  &           & Yes     &         & Yes      & Yes     & Yes      & Yes         &                             \\ \hline
2,4,4,6  &           & Yes     & Yes     &          & Yes     & Yes      &             &                             \\ \hline
4,4,4,4  & Yes       & Yes     &         &          & Yes     & Yes      & Yes         & Yes \\ \hline
\end{tabular}
\end{center}
\begin{center}
\begin{tabular}{|c|c|c|c|c|}
\hline
Even     & DDDH                        & DDDS                        & DHHH                        & DHHS               \\ \hline
1,3,3,9  &                             &                             &                             & Yes 		 \\ \hline
1,3,5,7  &                             &                             & Yes 			       & Yes              \\ \hline
1,5,5,5  &                             &                             & Yes                         &                 \\ \hline
3,3,3,7  &                             & Yes                         & Yes                         &                  \\ \hline
3,3,5,5  & Yes                         &                             & Yes                         & Yes                \\ \hline
\end{tabular}
\end{center}
    \caption{Possible value distribution for even and odd combinations.}
    \label{table:dist}
\end{table}

For two type combinations to be the combinations for the rows and columns of a semimagic quad square, they must share at least one value distribution in the list above. However, sharing the same value distribution does not mean that two combinations are a valid pair. For example, DDSS and HHHS both have 2,2,6,6 as a possible value distribution. However, when paired, the two S rows must have the same values because of the S column, which totals to at least 8 for one value. So 2,2,6,6 is impossible for DDSS and HHHS as a pair, even though they are valid distributions for DDSS and HHHS separately. Using the above list, we went through all the pairs that shared value distributions, checked to see if the corresponding square exists, and put the existing combinations into Table~\ref{table:pairs}.

\begin{table}[ht!]
\begin{center}
\begin{tabular}{|c|c|c|c|c|c|c|c|c|}
\hline
EVEN & DDDD        & DDHH                   & DDHS       & DDSS       & HHHH               & HHHS     & HHSS            & SSSS     \\ \hline
DDDD & 4,4,4,4     & 4,4,4,4                &            &            & 4,4,4,4            & 4,4,4,4  & 4,4,4,4          & 4,4,4,4  \\ \hline
DDHH &  & \begin{tabular}[c]{@{}c@{}}2,2,6,6\\ 2,4,4,6\\ 4,4,4,4\end{tabular} & 2,4,4,6                  & 2,2,6,6                  & \begin{tabular}[c]{@{}c@{}}2,2,6,6\\ 2,4,4,6\\ 4,4,4,4\end{tabular}           & \begin{tabular}[c]{@{}c@{}}2,2,6,6\\ 2,4,4,6\end{tabular} & 2,2,6,6         &          \\ \hline
DDHS &  &  & 2,2,4,8                  &     & 2,2,4,8        &                 &                  &          \\ \hline
DDSS &  &  &  &     &                & 2,2,2,10        &                  &          \\ \hline
HHHH &  &  &  &  & \begin{tabular}[c]{@{}c@{}}0,0,8,8\\ 0,4,4,8\\ 0,4,6,6\\ 4,4,4,4\end{tabular} & 0,4,4,8         & \begin{tabular}[c]{@{}c@{}}0,0,8,8\\ 0,4,4,8\end{tabular} & 0,0,8,8  \\ \hline
HHHS &  & &  &  &            & 0,0,6,10        &                &          \\ \hline
HHSS &  &  &  &  &           &            & 0,0,4,12         &          \\ \hline
SSSS &  & & &  &                & &           & 0,0,0,16 \\ \hline
\end{tabular}
\end{center}
\begin{center}
\begin{tabular}{|c|c|c|c|c|}
\hline
ODD  & DDDH                     & DDDS                     & DHHH           & DHHS                   \\ \hline
DDDH & 3,3,5,5                  &     & 3,3,5,5        & 3,3,5,5                  \\ \hline
DDDS & & 3,3,3,7                  & 3,3,3,7        &        \\ \hline
DHHH &  & & \begin{tabular}[c]{@{}c@{}}1,3,5,7\\ 1,5,5,5\\ 3,3,3,7\\ 3,3,5,5\end{tabular} & 1,3,5,7         \\ \hline
DHHS &  &  &                                & 1,3,3,9                   \\ \hline
\end{tabular}
\end{center}
    \caption{Type pairs and their possible value distributions for even and odd combinations.}
    \label{table:pairs}
\end{table}

In the next subsection, we study a particular case of an EvenQuads-16 deck.

\subsection{EvenQuads-16 deck}

Let one of the three attributes, say color, be the attribute that is the same across all of the cards of an EvenQuads-64 deck. Suppose this color is red. There are 16 red colors in the deck. So all of them participate in the square. This is equivalent to studying the EvenQuads-16 deck and also a grid with numbers from 0 to 15.

The semimagic quad squares of the two attributes other than color must have 4 of each value since all of the red cards participate. So the value distribution of the quad combinations for the rows and columns must be 4,4,4,4. Thus, we can only have DDDD, DDHH, HHHH, HHHS, HHSS, or SSSS types for rows or columns. This creates 36 pairs of types. As we can swap rows and columns, we can reduce this number to 21.

\begin{proposition}
If there exists an S column, then all rows are types D. 
\end{proposition}

\begin{proof}
If there is an S column, say 0000. Then all 0s are used, and we cannot have S or H rows.
\end{proof}

It follows that cases DDHH/HHHS, DDHH/HHSS, DDHH/SSSS, HHHH/HHHS, HHHH/HHSS, HHHH/SSSS, HHHS/HHHS, HHHS/HHSS, HHHS/SSSS, HHSS/HHSS, HHSS/SSSS, and SSSS/SSSS are impossible. This leaves the following 9 cases: DDDD/DDDD, DDDD/DDHH, DDDD/HHHH, DDDD/HHHS, DDDD/HHSS, DDDD/SSSS, DDHH/DDHH, DDHH/HHHH, and HHHH/HHHH. We study these examples below and provide the total number of possible squares up to symmetries: relabeling the values, shuffling rows and columns, and reflections of the square. We study these cases for one attribute only.

\textbf{Case DDDD/DDDD, 4 squares.} If all rows and all columns are D, then we have a Latin square. It is well-known that up to symmetries, there are four different ones \cite{Riordan}:

\begin{center}
\begin{tabular}{|c|c|c|c|}
\hline
0 & 1 & 2 & 3 \\ \hline
1 & 2 & 3 & 0 \\ \hline
2 & 3 & 0 & 1 \\ \hline
3 & 0 & 1 & 2 \\ \hline
\end{tabular}
\quad
\begin{tabular}{|c|c|c|c|}
\hline
0 & 1 & 2 & 3 \\ \hline
1 & 3 & 0 & 2 \\ \hline
2 & 0 & 3 & 1 \\ \hline
3 & 2 & 1 & 0 \\ \hline
\end{tabular}
\quad
\begin{tabular}{|c|c|c|c|}
\hline
0 & 1 & 2 & 3 \\ \hline
1 & 0 & 3 & 2 \\ \hline
2 & 3 & 1 & 0 \\ \hline
3 & 2 & 0 & 1 \\ \hline
\end{tabular}
\quad
\begin{tabular}{|c|c|c|c|}
\hline
0 & 1 & 2 & 3 \\ \hline
1 & 0 & 3 & 2 \\ \hline
2 & 3 & 0 & 1 \\ \hline
3 & 2 & 1 & 0 \\ \hline
\end{tabular} .
\end{center}

\textbf{Case DDDD/DDHH, 2 squares.} Suppose all rows are D. Without loss of generality, we can assume that the first row is 0123 and the first column is 0011 (type H):
\begin{center}
\begin{tabular}{|c|c|c|c|}
\hline
0 & 1 & 2 & 3 \\ \hline
0 & ? & ? & ? \\ \hline
1 & ? & ? & ? \\ \hline
1 & ? & ? & ? \\ \hline
\end{tabular} .
\end{center}
We need to put one more 1, and it has to go in the second row. It cannot be in the second column, as this would mean that the second column is the second H column, and the other 2 columns would be D. If that were the case, however, that would mean that there would be a fifth and sixth 1 in the last 2 columns, which is impossible. Without loss of generality, we can assume that 1 is in the third column. That means the second row is either 0213 or 0312. It follows that the fourth column must be the other H column, as the second and third both have 1s, which we cannot have more of. Moreover, the fourth column must have two 2s and two 3s. As we can still swap the last two rows, we can assume that the third row starts as 10. Thus we have the following two cases:
\begin{center}
\begin{tabular}{|c|c|c|c|}
\hline
0 & 1 & 2 & 3 \\ \hline
0 & 3 & 1 & 2 \\ \hline
1 & 0 & 3 & 2 \\ \hline
1 & 2 & 0 & 3 \\ \hline
\end{tabular}
\quad
\begin{tabular}{|c|c|c|c|}
\hline
0 & 1 & 2 & 3 \\ \hline
0 & 2 & 1 & 3 \\ \hline
1 & 0 & 3 & 2 \\ \hline
1 & 3 & 0 & 2 \\ \hline
\end{tabular} .
\end{center}

\textbf{Case DDDD/HHHH, 3 squares.} Suppose all rows are D and all columns are H. Without loss of generality, we can assume that the first column is 0011 and the first row is 0123. Moreover, the second column has to contain 1, and this 1 can only go in row 2. Thus, we have the following arrangement. Keep in mind that we can still flip the last two rows:
\begin{center}
\begin{tabular}{|c|c|c|c|}
\hline
0 & 1 & 2 & 3 \\ \hline
0 & 1 & ? & ? \\ \hline
1 & ? & ? & ? \\ \hline
1 & ? & ? & ? \\ \hline
\end{tabular} .
\end{center}

We have two cases for the second row: 0123 and 0132.

Case 1. Suppose the second row is 0123. In this case, the third and the fourth row must be the same. We can list options: 1032, 1230, 1302. Given that we can swap 2 and 3, we actually have two cases:
\begin{center}
\begin{tabular}{ |c|c|c|c| } 
 \hline
 0 & 1 & 2 & 3 \\ 
 \hline
 0 & 1 & 2 & 3 \\ 
 \hline
 1 & 0 & 3 & 2 \\ 
 \hline
 1 & 0 & 3 & 2 \\
 \hline
\end{tabular}
\quad
\begin{tabular}{ |c|c|c|c| } 
 \hline
 0 & 1 & 2 & 3 \\ 
 \hline
 0 & 1 & 2 & 3 \\ 
 \hline
 1 & 2 & 3 & 0 \\ 
 \hline
 1 & 2 & 3 & 0 \\
 \hline
\end{tabular} .
\end{center}

Case 2. Suppose the second row is 0132.
The last two columns will use all threes and twos. Thus, the second column has to be 1100. Without loss of generality, we have one case (as we can swap the last two rows):
\begin{center}
\begin{tabular}{ |c|c|c|c| } 
 \hline
 0 & 1 & 2 & 3 \\ 
 \hline
 0 & 1 & 3 & 2 \\ 
 \hline
 1 & 0 & 2 & 3 \\ 
 \hline
 1 & 0 & 3 & 2 \\
 \hline
\end{tabular} .
\end{center}

\textbf{Case DDDD/HHHS, 1 square.} Without loss of generality, let the first row be 0123, and the first column be S and 0000. There cannot be any more 0s, so the remaining values for the HHH columns are four 1s, 2s, and 3s. The only ways they can be split into three H columns is (1122, 2233, 1133) or (1133, 1122, 2233); however, (1133, 1122, 2233) can be achieved by switching the 2s and 3s in (1122, 2233, 1133), so they are indistinguishable. Without loss of generality, we assume that the values of the second column are in non-decreasing order, 1122. The 2s in the third column cannot be in the same rows as the 2s in the second column, so the third column is determined as 2233, followed by the fourth column, which is determined as 3311. Therefore the following square is the only square for this case:
\begin{center}
\begin{tabular}{|c|c|c|c|}
\hline
0 & 1 & 2 & 3 \\ \hline
0 & 1 & 2 & 3 \\ \hline
0 & 2 & 3 & 1 \\ \hline
0 & 2 & 3 & 1 \\ \hline
\end{tabular} .
\end{center}

\textbf{Case DDDD/HHSS, 1 square.} Without loss of generality, we can assume that the first row is 0123, and the first column is H and 0011. The second column is forced to be H because the first column already uses two 1s. Thus, the last two columns are both S:
\begin{center}
\begin{tabular}{|c|c|c|c|}
\hline
     0& 1 & 2 & 3 \\
\hline 
     0& 1 & 2 & 3 \\
\hline
     1 & 0 & 2 & 3 \\
\hline 
     1 & 0 & 2 & 3 \\
\hline 
\end{tabular} .
\end{center}

\textbf{Case DDDD/SSSS, 1 case.} If all rows are D and all columns are S, then we have one case:
\begin{center}
\begin{tabular}{|c|c|c|c|}
\hline
0 & 1 & 2 & 3 \\ \hline
0 & 1 & 2 & 3 \\ \hline
0 & 1 & 2 & 3 \\ \hline
0 & 1 & 2 & 3 \\ \hline
\end{tabular} .
\end{center}

\textbf{Case DDHH/DDHH, 3 squares.} Consider the intersection of two D rows and two D columns. We show that each value appears there no more than once. Without loss of generality, suppose a value 0 appears twice. Then it has to appear two more times in an H row and two more times in an H column, which is a contradiction.

Without loss of generality, the first row is 0123, and the first column is 0123. Let the types for the columns be DDHH. By the argument above, the second row cannot be type D, as the top left 2-by-2 box contains two 1s. Without loss of generality, we can assume that the third row is type D. From the argument above, we know that the value of the third row, second column must be 3. Thus, the third row has to be either 2301 or 2310. In any case, we can place the fourth 1 uniquely. We got the following cases:
\begin{center}
\begin{tabular}{|c|c|c|c|}
\hline
0 & 1 & 2 & 3 \\ \hline
1 & ? & ? & 1 \\ \hline
2 & 3 & 0 & 1 \\ \hline
3 & ? & ? & 3 \\ \hline
\end{tabular}
\quad
\begin{tabular}{|c|c|c|c|}
\hline
0 & 1 & 2 & 3 \\ \hline
1 & ? & 1 & ? \\ \hline
2 & 3 & 1 & 0 \\ \hline
3 & ? & 2 & ? \\ \hline
\end{tabular} .
\end{center}

The two remaining 0s must be placed in the same row because both remaining rows are type H. In the second case, the two 0s can only be placed in the second row since the values in the fourth row are already different, and there cannot be 3 different values in a row of type H. We find that placing the two 0s in the second row works. In the first case, the 0s can be put in either row, and both possibilities work. So in total, for DDHH/DDHH, there are 3 distinct squares that are shown below:
\begin{center}
\begin{tabular}{|c|c|c|c|}
\hline
0 & 1 & 2 & 3 \\ \hline
1 & 0 & 0 & 1 \\ \hline
2 & 3 & 0 & 1 \\ \hline
3 & 2 & 2 & 3 \\ \hline
\end{tabular}
\quad
\begin{tabular}{|c|c|c|c|}
\hline
0 & 1 & 2 & 3 \\ \hline
1 & 2 & 2 & 1 \\ \hline
2 & 3 & 0 & 1 \\ \hline
3 & 0 & 0 & 3 \\ \hline
\end{tabular}
\quad
\begin{tabular}{|c|c|c|c|}
\hline
0 & 1 & 2 & 3 \\ \hline
1 & 0 & 1 & 0 \\ \hline
2 & 3 & 1 & 0 \\ \hline
3 & 2 & 2 & 3 \\ \hline
\end{tabular} .
\end{center}

\textbf{Case DDHH/HHHH, 2 squares.} Without loss of generality, we can assume that the first row is D and 0123 and the first column is 0011. The second row cannot be D, as otherwise, 1 would already be used four times, and the other H row would be impossible. So, if the second row is H, without loss of generality, we can assume that it contains two 0s and two 2s. Moreover, the second row cannot be 0202, as this case makes 2 appear in three different columns, which is impossible. Thus, we have two cases:
\begin{center}
\begin{tabular}{|c|c|c|c|}
\hline
     0& 1 & 2 & 3 \\
\hline 
     0& 0 & 2 & 2 \\
\hline
     1 & ? & ? & ? \\
\hline 
     1 & ? & ? & ? \\
\hline 
\end{tabular}
\quad
\begin{tabular}{|c|c|c|c|}
\hline
     0& 1 & 2 & 3 \\
\hline 
     0& 2 & 2 & 0 \\
\hline
     1 & ? & ? & ? \\
\hline 
     1 & ? & ? & ? \\
\hline 
\end{tabular} .
\end{center}

Without loss of generality, we can assume that the last two values in the second column are in increasing order. After that, we can finish the two cases uniquely:
\begin{center}
\begin{tabular}{|c|c|c|c|}
\hline
     0& 1 & 2 & 3 \\
\hline 
     0& 0 & 2 & 2 \\
\hline
     1 & 0 & 3 & 2 \\
\hline 
     1 & 1 & 3 & 3 \\
\hline 
\end{tabular}
\quad
\begin{tabular}{|c|c|c|c|}
\hline
     0& 1 & 2 & 3 \\
\hline 
     0& 2 & 2 & 0 \\
\hline
     1 & 1 & 3 & 3 \\
\hline 
     1 & 2 & 3 & 0 \\
\hline 
\end{tabular} .
\end{center}

\textbf{Case HHHH/HHHH, 3 squares.} Without loss of generality, we can assume that the first row is 0011. The other two 1s must lie in the third and fourth columns. Similarly, the other two 0s must lie in the first and second columns. Without loss of generality, assume that the first column is 0022. Now that we have placed three of the 0s, the last zero must be in the second row and second column:
\begin{center}
\begin{tabular}{|c|c|c|c|}
\hline
0 & 0 & 1 & 1 \\ \hline
0 & 0 & ? & ? \\ \hline
2 & ? & ? & ? \\ \hline
2 & ? & ? & ? \\ \hline
\end{tabular} .
\end{center}

The two missing 1s must be in the same row. Moreover, they need to be in columns 3 and 4. Therefore, we have 2 cases: either the ones are in the second row or, without loss of generality, in the third row:
\begin{center}
\begin{tabular}{|c|c|c|c|}
\hline
0 & 0 & 1 & 1 \\ \hline
0 & 0 & 1 & 1 \\ \hline
2 & ? & ? & ? \\ \hline
2 & ? & ? & ? \\ \hline
\end{tabular}
\quad
\begin{tabular}{|c|c|c|c|}
\hline
0 & 0 & 1 & 1 \\ \hline
0 & 0 & ? & ? \\ \hline
2 & ? & 1 & 1 \\ \hline
2 & ? & ? & ? \\ \hline
\end{tabular} .
\end{center}

For the first case above, there are only two ways, without loss of generality, to fill in the remaining two 2s because they must be in the same column. There is one way for the second case above. The remaining spots must all be 3s, so this will result in three cases:
\begin{center}
\begin{tabular}{|c|c|c|c|}
\hline
0 & 0 & 1 & 1 \\ \hline
0 & 0 & 1 & 1 \\ \hline
2 & 2 & 3 & 3 \\ \hline
2 & 2 & 3 & 3 \\ \hline
\end{tabular}
\quad
\begin{tabular}{|c|c|c|c|}
\hline
0 & 0 & 1 & 1 \\ \hline
0 & 0 & 1 & 1 \\ \hline
2 & 3 & 2 & 3 \\ \hline
2 & 3 & 2 & 3 \\ \hline
\end{tabular}
\quad
\begin{tabular}{|c|c|c|c|}
\hline
0 & 0 & 1 & 1 \\ \hline
0 & 0 & 3 & 3 \\ \hline
2 & 2 & 1 & 1 \\ \hline
2 & 2 & 3 & 3 \\ \hline
\end{tabular} .
\end{center}

We found 20 different semimagic quad squares up to symmetries in the EvenQuads-16 deck.

\section{Strongly Magic Quad Squares}
\label{sec:stronglymagicsquares}

\subsection{The number of strongly magic quad squares}

\begin{theorem}
There are $2^n(2^n-1)(2^n-2)(2^n-4)(2^n-8)$ strongly magic quad squares that can be made by using the cards from the EvenQuads-$2^n$ deck.
\end{theorem}

\begin{proof}
\label{thm:smsnumber}
By Lemma~\ref{lemma:symmetrymultiplier}, we need to calculate the number of strongly magic quad squares of type C. But given the first row and first column, the strongly magic quad square is uniquely defined. Thus, the theorem follows.
\end{proof}

There are $64(64-1)(64-2)(64-4)(64-8)=839946240$ strongly magic quad squares that can be made by using the cards from the EvenQuads deck. For the EvenQuads-$2^n$ deck, where $n \geq 4$, the number of strongly magic quad squares is provided by the following sequence, which is now sequence A362874:
\[322560,\ 19998720,\ 839946240,\ 30478049280,\ 1036253675520,\ 34162943754240,\ \ldots.\]

The first term corresponds to the EvenQuads-16 deck. All such squares use the same 16 cards. Thus, the number of such squares is the number of ways to shuffle the entries of a strongly magic quad square to produce another strongly magic quad square. Since all entries are distinct, no squares are fixed by the shuffling, so every equivalence class of squares in the EvenQuads-$2^n$ deck under shuffling has the same number of elements. This means that the number of strongly magic quad squares in the EvenQuads-$2^n$ deck is divisible by the number of strongly magic quad squares in the EvenQuads-16 deck. After dividing, we get 
\[\frac{2^{5n} - 15\cdot 2^{4n}+ 70\cdot 2^{3n} - 120 \cdot 2^{2n} + 64 \cdot 2^n}{322560},\]
which corresponds to one of the definitions of sequence A308436:
\[1,\ 62,\ 2604,\ 94488,\ 3212592,\ 105911904,\ 3439615168,\ 110880192896,\ \ldots.\]
	
A strongly magic quad square forms a tesseract in our affine space. We can view a tesseract as an EvenQuads-16 deck. Suppose we arrange the EvenQuads-16 deck as a strongly magic quad square. Consider a linear function from an EvenQuads-16 deck to an EvenQuads-64 deck. Such a function preserves the property of four cards forming a quad. Then if we arrange the images of this function the same way as we arranged the initial EvenQuads-16 deck, we get a strongly magic quad square.

In a strongly magic quad square, the quad types of the first row and first column determine the types of the other rows and columns. We can prove that each row must be the same type of quad and similarly for each column. Let the cards on the first row from left to right be $a_1$, $a_2$, $a_3$, and $a_4$, and the cards on the second row from left to right be $b_1$, $b_2$, $b_3$, and $b_4$ Let $b_1=a_1+d$. We know from a strongly magic quad square's definition that $a_1+a_k+b_1+b_k = 0$, for $k \in \{2,3,4\}$. By substituting $b_1$, we get $a_1+a_k+a_1+d+b_k = 0$, which is equivalent to $a_k+d+b_k = 0$. Therefore, $b_k=a_k+d$.

Thus, we can see that the second row is the first row plus $d$. So, the second row is parallel to the first row. We can do the same for the other rows, and the same applies to columns as well. Since parallel planes have points in the same position in relation to the other points in the plane, all rows must be of the same type, and all columns must be of the same type.

Restricting our Table~\ref{table:pairs} to cases where rows/columns are the same type, we get the following Table~\ref{table:SMSpairs}.
\begin{table}[ht!]
\begin{center}
\begin{tabular}{|c|c|c|c|}
\hline
     & DDDD & HHHH                                                              & SSSS    \\ \hline
DDDD & 4,4,4,4 & 4,4,4,4                                                        & 4,4,4,4    \\ \hline
HHHH &      & \begin{tabular}[c]{@{}c@{}}0,0,8,8\\ 0,4,4,8\\ 0,4,6,6\\ 4,4,4,4\end{tabular} & 0,0,8,8    \\ \hline
SSSS &      &                                                                   & 0,0,0,16 \\ \hline
\end{tabular}
\end{center}
\caption{Potential types and distributions.}
\label{table:SMSpairs}
\end{table}

The table suggests the following squares might be possible: DDDD/DDDD, DDDD/HHHH, DDDD/SSSS, and HHHH/HHHH with distribution 4,4,4,4. Also, HHHH/SSSS with distribution 0,0,8,8, and SSSS/SSSS with distribution 0,0,0,16. In addition, we have HHHH/HHHH with distributions 0,0,8,8, 0,4,4,8, and 0,4,6,6. However, the last two are impossible.

Here is the explanation of why the last two distributions are impossible: they use exactly three different values. If we pick any three cards with three different values in a square, they have to be completed into a quad. The completion card has to have the fourth value, which creates a contradiction.

\subsection{Classification of strongly magic quad squares in the EvenQuads-16 deck}

In this section, for a given attribute, we classify all possible strongly magic quad squares up to symmetries. The symmetries are: relabeling the values, shuffling rows and columns, and rotations and reflections of the square.

\begin{theorem}
Table~\ref{table:SMSClass} shows all possible strongly magic quad squares for one attribute up to symmetries using the EvenQuads-16 deck.
\end{theorem}

\begin{table}[ht!]
\begin{center}
\begin{tabular}{|c|c|c|c|}
\hline
0 & 1 & 2 & 3 \\ \hline
1 & 0 & 3 & 2 \\ \hline
2 & 3 & 0 & 1 \\ \hline
3 & 2 & 1 & 0 \\ \hline
\end{tabular}
\quad
\begin{tabular}{|c|c|c|c|}
\hline
0 & 1 & 2 & 3 \\ \hline
0 & 1 & 2 & 3 \\ \hline
1 & 0 & 3 & 2 \\ \hline
1 & 0 & 3 & 2 \\ \hline
\end{tabular}
\quad
\begin{tabular}{|c|c|c|c|}
\hline
0 & 1 & 2 & 3 \\ \hline
0 & 1 & 2 & 3 \\ \hline
0 & 1 & 2 & 3 \\ \hline
0 & 1 & 2 & 3 \\ \hline
\end{tabular}
\end{center}
\begin{center}
\begin{tabular}{|c|c|c|c|}
\hline
0 & 0 & 1 & 1 \\ \hline
0 & 0 & 1 & 1 \\ \hline
1 & 1 & 0 & 0 \\ \hline
1 & 1 & 0 & 0 \\ \hline
\end{tabular}
\quad
\begin{tabular}{|c|c|c|c|}
\hline
0 & 0 & 1 & 1 \\ \hline
0 & 0 & 1 & 1 \\ \hline
2 & 2 & 3 & 3 \\ \hline
2 & 2 & 3 & 3 \\ \hline
\end{tabular}
\quad
\begin{tabular}{|c|c|c|c|}
\hline
0 & 0 & 1 & 1 \\ \hline
0 & 0 & 1 & 1 \\ \hline
0 & 0 & 1 & 1 \\ \hline
0 & 0 & 1 & 1 \\ \hline
\end{tabular}
\quad
\begin{tabular}{|c|c|c|c|}
\hline
0 & 0 & 0 & 0 \\ \hline
0 & 0 & 0 & 0 \\ \hline
0 & 0 & 0 & 0 \\ \hline
0 & 0 & 0 & 0 \\ \hline
\end{tabular}
\end{center}
\caption{Strongly magic quad squares in the EvenQuads-16 deck.}
\label{table:SMSClass}
\end{table}

\begin{proof}
\textbf{Rows are D, columns are D:} Without loss of generality, we can assume that the first row is 0123 and the first column is 0123. We can complete this square uniquely; see the first square in the table.

\textbf{Rows are D, columns are H:} Without loss of generality, we can assume that the first row is 0123 and the first column is 0011. We can complete this square uniquely. See the second square in the table.

\textbf{Rows are D, columns are S:} After relabeling, we can assume that our square has all rows 0123. See the third square in the table.

\textbf{Rows are H, columns are H:} Without loss of generality, we can assume that the first row is 0011. The first column has to have two zeros. The other value in the first column could be either the same as the other value in the first row or not. Thus, we have two cases, which are the fourth and fifth in the table.

\textbf{Rows are H, columns are S:} Without loss of generality, we can assume that the first row is 0011 and the first column is 0000. After completing the square, we get the sixth square in the table.

\textbf{Rows are S, columns are S:} When all rows and columns are S, there is only one value of the given attribute. That means all the cards of this value are used. Without loss of generality, the square is the last one in the table.
\end{proof}

If one attribute is the same, we have 16 cards in total. From Theorem~\ref{thm:smsnumber}, the total number of strongly magic quad squares is, in this case, $16 \cdot 15 \cdot 14 \cdot 12 \cdot 8 = 322560$.

Each value has to appear four times. That means we can look at the list above. The corresponding cases are DS, DD, DH, and the second case for HH.

\section{The Number of Semimagic Quad Squares}
\label{sec:numbersmsquares}

In this section, we calculate the number of semimagic quad squares for a given EvenQuads-$2^n$ deck. As before, we enumerate the cards with binary strings. Moreover, we convert strings to numbers 0 through $2^n-1$.

By Lemma~\ref{lemma:symmetrymultiplier}, we need to calculate the number of semimagic quad squares of type C. We checked by a program that for an EvenQuads-16 deck, the number of semimagic quad squares of type C is 112. Thus, the total number of semimagic quad squares for the EvenQuads-16 deck has to be 
\[112 \cdot 16 \cdot 15 \cdot 14 \cdot 12 \cdot 8 = 36126720.\]

Now we need to count semimagic quad squares of type C for any deck. Every type C square can be decomposed into a sum of a semimagic quad square that consists of numbers below 16 and the other one of numbers divisible by 16. For example
\begin{center}
\begin{tabular}{|c|c|c|c|}
\hline
0 & 1 & 2 & 3 \\ \hline
4 & 17 & 32 & 53 \\ \hline
8 & 5 & 6 & 11 \\ \hline
12 & 21 & 36 & 61 \\ \hline
\end{tabular}
= 
\begin{tabular}{|c|c|c|c|}
\hline
0 & 1 & 2 & 3 \\ \hline
4 & 1 & 0 & 5 \\ \hline
8 & 5 & 6 & 11 \\ \hline
12 & 5 & 4 & 13 \\ \hline
\end{tabular}
+
\begin{tabular}{|c|c|c|c|}
\hline
0 & 0 & 0 & 0 \\ \hline
0 & 16 & 32 & 48 \\ \hline
0 & 0 & 0 & 0 \\ \hline
0 & 16 & 32 & 48 \\ \hline
\end{tabular} .
\end{center}

The first semimagic quad square in the sum has to have the first row 0, 1, 2, and 3, the first column 0, 4, 8, and 12, and all the numbers not exceeding 16. But it can have repeat values. We call such a square \textit{type A}. The second semimagic quad square has to have its first row and column all zeros and other values divisible by 16. It can have repeat integers too. Let us call such a semimagic quad square \textit{type B}. Thus, any square of type C is a sum of squares of types A and B.

However, a sum of squares of type A and type B might have repeated integers, thus, might not produce a square of type C. But how can we get repeated numbers in C? Since square A has integers below 16, and square B has integers divisible by 16, the only way square C can have two repeated integers in places $p_1$ and $p_2$ is if both A and B have repeated integers in places $p_1$ and $p_2$. 

We start by discussing possible patterns of repeated numbers in a square of type B. We denote the value in row $i$ and column $j$ of square $A$, respectively $B$, by $a_{i,j}$, respectively $b_{i,j}$. For example, consider the decomposition example above. The pattern of repeated numbers is the following: $b_{2,2} = b_{4,2} = 16$, $b_{2,3} = b_{4,3} = 32$, $b_{2,4} = b_{4,4} =48$, and all other numbers are equal to zero. It follows that for square C not to have repeated numbers, square A needs to satisfy the following conditions: $a_{2,2} \neq a_{4,2}$, $a_{2,3} \neq a_{4,3}$, $a_{2,4} \neq a_{4,4}$, and numbers $a_{3,2}$, $a_{3,3}$, and $a_{3,4}$ are all different and differ from the first row/column of square A, that is they are not in the set $\{0,1,2,3,4,8,12\}$. We call the set of pairs of repeat integers in square B, \textit{a repeat pattern of B}.

Given square $B$ of type B, we call square $A$ of type A a \textit{legit pair of $B$} if $A+B$ does not have repeated numbers. 

In the next lemma, we show that we can look at repeat patterns up to symmetries.

\begin{lemma}
If two squares $B_1$ and $B_2$ of type B have the same repeat pattern up to shuffling rows/columns and reflections along the main diagonal, then the number of legit pairs for $B_1$ is the same as the number of legit pairs of $B_2$.
\end{lemma}

\begin{proof}
Consider the only square of type C in the EvenQuads-16 deck. Every transformation $t$ of this square that includes shuffling rows/columns and reflections along the main diagonal can be expressed as an affine transformation $a$ in our space that does not change $B_i$. For example, swapping the second and the third columns corresponds to swapping the last two digits in the binary representation of numbers representing our cards. Similarly, reflection across the main diagonal corresponds to swapping the third to last and the last digits and the second to last and the fourth to last digits. Other transformations are similar.

Suppose we get square $B_2$ after applying $t$ to $B_1$: $B_2 = t(B_1)$. Consider $A_1$, which is a legit pair for $B_1$. Then $t(A_1) + t(B_1) = t(A_1) + B_2$ does not have repeated numbers, making $t(A_1)$ potentially a legit pair of $B_2$, except that $t(A_1)$ might not be of type C, which is a necessary condition for a legit pair. Let $a$ be the affine transformation in our space that does not change squares of type B while changing $t(A_1)$ to type C. We apply these transformation on both $t(A_1)$ and $B_2$: $at(A_1) = A_2$ and $aB_2 = B_2$. Thus, we found a one-to-one correspondence between legit pairs of $B_1$ and $B_2$.
\end{proof}

The lemma above allows us to study squares of type B up to symmetries (row/column shuffling and reflections). Given square $B_1$ of type B, we denote $F(B_1)$, \textit{the symmetry factor} of $B_1$: the number of different repeat patterns we can get from $B_1$ by shuffling the last three rows and columns and reflecting along the main diagonal.

Now we study squares of type B$'$, which are the squares of type B divided by 16. We can multiply everything by 16 later. We want to classify such squares up to relabeling, shuffling rows/columns, and reflections along the main diagonal.

\begin{lemma}
Table~\ref{table:B'} provides all possible squares of type B$'$ up to affine transformations, permutations of the last three rows/columns, and reflections along the main diagonal.
\end{lemma}

\begin{table}[ht!]
\begin{center}
\begin{tabular}{|c|c|c|c|}
\hline
0&0&0&0\\\hline
0&0&0&0\\\hline
0&0&0&0\\\hline
0&0&0&0 \\ \hline
\end{tabular}
\quad
\begin{tabular}{|c|c|c|c|}
\hline
0&0&0&0\\\hline
0&0&0&0\\\hline
0&0&1&1\\\hline
0&0&1&1 \\ \hline
\end{tabular}
\quad
\begin{tabular}{|c|c|c|c|}
\hline
0&0&0&0\\\hline
0&0&0&0\\\hline
0&1&2&3\\\hline
0&1&2&3 \\ \hline
\end{tabular}
\quad
\begin{tabular}{|c|c|c|c|}
\hline
0&0&0&0\\\hline
0&0&1&1\\\hline
0&1&0&1\\\hline
0&1&1&0 \\ \hline
\end{tabular}
\quad
\begin{tabular}{|c|c|c|c|}
\hline
0&0&0&0\\\hline
0&0&1&1\\\hline
0&1&2&3\\\hline
0&1&3&2 \\ \hline
\end{tabular}
\end{center}
\begin{center}
\begin{tabular}{|c|c|c|c|}
\hline
0&0&0&0\\\hline
0&0&1&1\\\hline
0&2&0&2\\\hline
0&2&1&3 \\ \hline
\end{tabular}
\quad
\begin{tabular}{|c|c|c|c|}
\hline
0&0&0&0\\\hline
0&0&1&1\\\hline
0&2&4&6\\\hline
0&2&5&7 \\ \hline
\end{tabular}
\quad
\begin{tabular}{|c|c|c|c|}
\hline
0&0&0&0 \\ \hline
0&1&2&3 \\ \hline
0&2&3&1 \\ \hline
0&3&1&2 \\ \hline
\end{tabular}
\quad
\begin{tabular}{|c|c|c|c|}
\hline
0&0&0&0 \\ \hline
0&1&2&3 \\ \hline
0&2&4&6 \\ \hline
0&3&6&5 \\ \hline
\end{tabular}
\quad
\begin{tabular}{|c|c|c|c|}
\hline
0&0&0&0 \\ \hline
0&1&2&3 \\ \hline
0&4&8&12 \\ \hline
0&5&10&15 \\ \hline
\end{tabular}
\end{center}
\caption{Squares of type B$'$.}
\label{table:B'}
\end{table}

\begin{proof}
We start with squares that contain at least 8 zeros. Because of symmetries, we can assume that $a_{2,2} = 0$. Then, without loss of generality, the second row can be either 0000 or 0011. We look at these cases separately.

\textbf{The second row is 0000.} Then, without loss of generality, the third row is one of 0000, 0011, and 0123. These are cases 1, 2, and 3 in the table.

\textbf{The second row is 0011.} We also assume that the second column is not zero; otherwise, by reflection, it would have corresponded to the previous cases. Without loss of generality, the second column is either 0011 or 0022. 

\begin{itemize}
\item \textbf{The second column is 0011.} The third row can either be type D or H. If it is type H, then, remembering we can swap the last two columns, we get case 4 in the table. If the third row is D, then we get case 5.
\item \textbf{The second column is 0022.} The third row can again be type H or D. If it is type H, then again, by swapping the third and fourth columns, we can place 0 on the diagonal, and we get case 6. If the third row is type D, we consider the value $a_{3,3}$, which cannot be 0 or 2. Suppose it is 1 or 3, then we get case 6 after shuffling rows and columns. If $a_{3,3}$ is greater than 3, without loss of generality, we can assume that it is 4 and get case 7. 
\end{itemize}

We continue with cases with 7 zeros. For all the rows and columns other than the first row and the first column, if they are type H or S, we will have more than 7 zeros which is not allowed. Therefore, all these rows and columns are type D. Without loss of generality, let the second row be 0, 1, 2, 3. The second column is either in the same plane or not. If it is in the same plane, we can assume, without loss of generality, that it is 0, 1, 2, 3. If the second column is not in the same plane, we can assume, without loss of generality, that $a_{3,2} = 4$. Thus, the second column is 0, 1, 4, 5. In any case, after that, the square is uniquely defined by its value $a_{3,3}$ in row 3, column 3. We proceed with the two cases we identified.

\textbf{The second column is 0, 1, 2, 3.} If the third row is in the same plane as both the second row and column, then $a_{3,3}$ has to be 1 or 3. If it is 1, we get that $a_{4,4}$ must be 0, which is not allowed. If $a_{3,3}$ is 3, we get case 8 in Table~\ref{table:B'}. If the third row is not in the same plane as both the second row and column, then, without loss of generality, we can assume $a_{3,3}$ is 4. We get that $a_{3,4}$ and $a_{4,3}$ must be 6 and that $a_{4,4}$ is 5. This is case 9.

\textbf{The second column is 0, 1, 4, 5.} Trying values of $a_{3,3}$, we first check what happens if $a_{3,3}$ is in the same hyperplane as the one generated by the second rows and columns: we check values from 0 to 7. Note that we must have no columns that are a permutation of 0, 1, 2, 3, as otherwise, after shuffling the rows and columns, we can get cases 8 or 9, which we already covered. There cannot be more 0s, so $a_{3,3}$ cannot be 0. Also, $a_{3,3}$ cannot be 1 or 3, or the third column would be a permutation of 0, 1, 2, 3. It cannot be 2 or 4 because it would make the third row or third column type H, which means including another 0. It cannot be 5 or 6 because $a_{4,3}$ would be either 1 or 2, which completes the last column into a permutation of 0, 1, 2, 3. Finally, it cannot be 7 or $a_{4,3}$ is 3, which makes the last column type H. So we do not get a case when $a_{3,3}$ is in the same hyperplane. If $a_{3,3}$ is not in the same hyperplane, then, without loss of generality, it is 8. Then, $a_{3,4}$ and $a_{4,3}$ are 10 and 12 and $a_{4,4}$ is 15. This is the last case in the table.
\end{proof}

Now we calculate the number of squares of type C for each case.

\begin{lemma}
\label{lemma:typeC}
The last column in Table~\ref{table:summary} shows the number of semimagic quad squares of type C for each case. The total number of semimagic quad squares of type C for the EvenQuads-$2^n$ deck is
\begin{align*}
112 + 2823  (2^n - 16) + 2531 & (2^n - 16) (2^n - 32) + 159 (2^n - 16) (2^n - 32) (2^n-64) + \\
  &(2^n - 16) (2^n - 32)(2^n-64)(2^n-128).
\end{align*}
\end{lemma}
\begin{table}[ht!]
\begin{center}
\begin{tabular}{|c|c|c|c|c|}
\hline
 & F & Mult & \# A & \# C \\
\hline
1	& 1	& 1 			& 112 				& 112	\\ \hline
2	& 9	& $M-1$ 		& $16 \cdot 119$ 		& $1071(2^n-16)$	\\ \hline
3	& 6	& $(M-1)(M-2)$ 	& $16^2 \cdot 42$ 	& $252(2^n-16)(2^n-32)$	\\ \hline
4	& 6	& $(M-1)$ 		& $16 \cdot 292$ 		& $1752(2^n-16)$	\\ \hline
5	& 9	& $(M-1)(M-2)$ 	& $16^2 \cdot 95$ 	& $855(2^n-16)(2^n-32)$	\\ \hline
6	& 18	& $(M-1)(M-2)$ 	& $16^2 \cdot 64$ 	& $1152(2^n-16)(2^n-32)$	\\ \hline
7	& 9	& $(M-1)(M-2)(M-4)$ 	& $16^3 \cdot 9$ 	& $81(2^n-16)(2^n-32)(2^n-64)$	\\ \hline
8	& 2	& $(M-1)(M-2)$ 	& $16^2 \cdot 136$ 	& $272(2^n-16)(2^n-32)$	\\ \hline
9	& 6	& $(M-1)(M-2)(M-4)$ 	& $16^3 \cdot 13$ 	& $78(2^n-16)(2^n-32)(2^n-64)$	\\ \hline
10	& 1	& $(M-1)(M-2)(M-4)(M-8)$ 	& $16^4$ 	& $(2^n-16)(2^n-32)(2^n-64)(2^n-128)$	\\ \hline
\end{tabular}
\end{center}
\caption{Calculating squares of type C.}
\label{table:summary}
\end{table}

\begin{proof}
The first column in the table is the case number. The second column is the symmetry factor. The calculation of the symmetry factor is straightforward. We provide the explanation for the most interesting cases.

For the second case, the symmetry factor is 9: we can choose the two nontrivial rows in three ways and two nontrivial columns in three ways. For the fourth case, the symmetry factor is 6: we are choosing three cells in a 3-by-3 square such that none of the chosen cells are in the same row or column. For the sixth case, the symmetry factor is 18: we choose the rows that contain zeros in 3 ways and then place zeros in these rows in 6 ways. For the ninth case, the symmetry factor is 6: we can choose places where we put the three distinct numbers in 6 ways; the pattern is then uniquely defined.

The third column represents the number of different squares of type $B'$ for a given case. We assume that the deck size is $M$. If a square has the only non-zero value 1, then without loss of generality, we can choose this value in $M-1$ ways. If the square has, in addition, value two, then we can choose this second value in $M-2$ ways. Then the value 3 is calculated from the bitwise XOR. Then, if the square also has 4, it corresponds to the factor $M-4$. Then the values below 8 are uniquely placed in the square. If there is value 8, it could be chosen in $M-8$ ways.

Column 4 shows the number of squares of type A, that match the restrictions imposed by the pattern in the corresponding case. This calculation was done by a program.

Column 5 shows the number of squares of type C for each case, where we substitute $\frac{2^n}{16}$ for $M$.

Summing up the last column, we get the total.
\end{proof}

For the standard deck with 64 cards, the number of semimagic quad squares of type C is 4023232. The sequence enumerating semimagic quad squares of type C for the EvenQuads-$2^n$ decks, for $n \geq 4$, which is now sequence A362963:
\[112,\ 45280,\ 4023232,\ 136941952,\ 3099135232,\ 58520273920,\ 1015268864512,\ \ldots.\]

Plugging in the number of semimagic quad squares of type C from Lemma~\ref{lemma:typeC} into the formula from Lemma~\ref{lemma:symmetrymultiplier}, we get the following enumeration of semimagic quad squares.

\begin{theorem}
\label{thm:nsms}
In the EvenQuads-$2^n$ deck, there are
\begin{align*}
2^n(2^n-1)&(2^n-2)(2^n-4)(2^n-8)\cdot \\
& (112 + 2823  (2^n - 16) + 2531 (2^n - 16) (2^n - 32) + 159 (2^n - 16) (2^n - 32) (2^n-64) + \\
  &(2^n - 16) (2^n - 32)(2^n-64)(2^n-128)).
\end{align*}
semimagic quad squares. 
\end{theorem}

For the standard deck with 64 cards, the number of semimagic quad squares is:
\[4023232 \cdot 64 \cdot 63 \cdot 62 \cdot 60 \cdot 56 = 3379298591047680.\]

The sequence enumerating semimagic quad squares for EvenQuads-$2^n$ decks, for $n \geq 4$ is now sequence A362964:
\[36126720,\ 905542041600,\ 3379298591047680,\ 4173723561555394560,\ 3211490275093527920640,\ \ldots.\]

The coefficients in the expression in Theorem~\ref{thm:nsms} represent the number of semimagic quad squares up to affine transformations such that the cards span a hyperplane of a particular dimension. Thus, the number of equivalence classes of semimagic quad squares in the EvenQuads-$2^n$ deck up to affine transformations is given by the following sequence that stabilizes. The sequence starts at index 4:
\[112,\ 2935,\ 5466,\ 5625,\ 5626,\ 5626,\ 5626, \ldots.\]

\section{The Number of Magic Quad Squares}
\label{sec:numbermsquares}

Magic squares have much fewer symmetries, but the code that calculates the number of magic squares is faster. Consider magic squares of type C. We know the first row and column of it. If we know the middle square: $c_{2,2}$, $c_{2,3}$, $c_{3,2}$, and $c_{3,3}$, we can calculate the rest of it. However, these parameters are dependent: we know that bitwise XOR of 3, 12, $c_{2,3}$, and $c_{3,2}$ has to be zero. So, for practical purposes, we have three parameters. We wrote a program to calculate the number of magic squares of type C for decks of sizes 16, 32, 64, and 128 and got the following numbers: 10, 1370, 70138, and 1159994.

The number of magic squares of type C for the EvenQuads-16 deck is small enough, so we list all of them in Table~\ref{table:mstypeC}.

\begin{table}[ht!]
\begin{center}
\begin{tabular}{|c|c|c|c|}
\hline
    0 & 1 & 2&3 \\
        \hline
    4&5&6&7 \\
        \hline
    8&9&10&11 \\
        \hline
    12&13&14&15 \\
        \hline
\end{tabular}
\quad
\begin{tabular}{|c|c|c|c|}
    \hline
    0&1&2&3 \\
        \hline
    4&5&6&7 \\
        \hline
    8&9&11&10 \\
        \hline
    12&13&15&14\\
    \hline
\end{tabular}
\quad
\begin{tabular}{|c|c|c|c|}
    \hline
    0&1&2&3 \\
        \hline
    4&5&6&7 \\
        \hline
    8&9&14&15 \\
        \hline
    12&13&10&11 \\
        \hline
\end{tabular}
\quad
\begin{tabular}{|c|c|c|c|}
    \hline
    0&1&2&3 \\
        \hline
    4&5&6&7 \\
        \hline
    8&9&15&14 \\
        \hline
    12&13&11&10\\
    \hline
\end{tabular}
\end{center}
\begin{center}
\begin{tabular}{|c|c|c|c|}
    \hline
    0&1&2&3 \\
    \hline
    4&6&5&7 \\
    \hline
    8&10&9&11 \\
    \hline
    12&13&14&15 \\
    \hline
\end{tabular}
\quad
\begin{tabular}{|c|c|c|c|}
 \hline
    0&1&2&3 \\
     \hline
    4&7&6&5 \\
     \hline
    8&9&10&11 \\
     \hline
    12&15&14&13 \\
    \hline
\end{tabular}
\quad
\begin{tabular}{|c|c|c|c|}
\hline
    0&1&2&3 \\
    \hline
    4&9&10&7 \\
    \hline
    8&5&6&11 \\
    \hline
    12&13&14&15\\
    \hline
\end{tabular}
\quad
\begin{tabular}{|c|c|c|c|}
\hline
    0&1&2&3 \\
    \hline
    4&10&9&7 \\
    \hline
    8&6&5&11 \\
    \hline
    12&13&14&15\\
    \hline
\end{tabular}
\end{center}
\begin{center}
\begin{tabular}{|c|c|c|c|}
\hline
    0&1&2&3 \\
    \hline
    4&13&6&15\\
    \hline
    8&9&10&11 \\
    \hline
    12&5&14&7\\
    \hline
\end{tabular}
\quad
\begin{tabular}{|c|c|c|c|}
\hline
    0&1&2&3 \\
     \hline
    4&15&6&13 \\
     \hline
    8&9&10&11 \\
     \hline
    12&7&14&5\\
    \hline
\end{tabular}
\end{center}
\caption{Magic squares of type C for the EvenQuads-16 deck.}
\label{table:mstypeC}
\end{table}

The data we received is enough to derive a formula for any deck.

\begin{lemma}
\label{lemma:typeCmagic}
The total number of magic quad squares of type C for the EvenQuads-$2^n$ deck is
\[10 + 85(2^n - 16) + 43(2^n - 16) (2^n - 32) + (2^n - 16) (2^n - 32) (2^n-64).\]
\end{lemma}

\begin{proof}
After the first row and column are fixed, the rest of the magic quad square is defined by the values $c_{2,2}$, $c_{2,3}$, $c_{3,2}$, and $c_{3,3}$. We know that $c_{3,2} = 15 \textrm{ XOR } c_{2,3}$. Thus, the square depends on three parameters: $c_{2,2}$, $c_{2,3}$, and $c_{3,3}$. That means the formula for the number of magic squares has to have the following format.
\[q_1 + q_2(2^n - 16) + q_3(2^n - 16) (2^n - 32) + q_4(2^n - 16) (2^n - 32) (2^n-64),\]
where $q_i$ are integer coefficients. Plugging in known values from the program, we get the desired result.
\end{proof}

The sequence of magic squares of type C for the EvenQuads-$2^n$ deck, where $n \geq 4$ is now sequence A361495:
\[10,\ 1370,\ 70138,\ 1159994,\ 12654010,\ 116939450,\ 1003021498,\ 8303802554,\ 67568410810,\ \ldots.\]

Plugging in the number of magic quad squares of type C from Lemma~\ref{lemma:typeCmagic} into the formula from Lemma~\ref{lemma:symmetrymultiplier}, we get the following enumeration of magic quad squares.

\begin{theorem}
\label{thm:nms}
In the EvenQuads-$2^n$ deck, there are
\begin{align*}
2^n(2^n-1)&(2^n-2)(2^n-4)(2^n-8)\cdot \\
& (10 + 85(2^n - 16) + 43(2^n - 16) (2^n - 32) + (2^n - 16) (2^n - 32) (2^n-64)).
\end{align*}
magic quad squares. 
\end{theorem}

The sequence enumerating magic quad squares for the EvenQuads-$2^n$ decks, for $n \geq 4$ is now sequence A361613:
\[3225600,\ 27398246400,\ 58912149381120,\ 35354354296504320,\ 13112764372566835200, \ldots.\]

The coefficients in the expression in Theorem~\ref{thm:nms} represent the number of magic quad squares up to affine transformations such that the cards span a hyperplane of a particular dimension. Thus, the number of equivalence classes of magic quad squares in the EvenQuads-$2^n$ deck up to affine transformations is given by the following sequence that stabilizes. The sequence starts at index 4:
\[10,\ 95,\ 138,\ 139,\ 139,\ 139,\ 139, \ldots.\]

\section{Acknowledgments}

We are grateful to the MIT PRIMES STEP program and its director, Slava Gerovitch, for allowing us the opportunity to do this research.


\begin{thebibliography}{9}

\bibitem{LAGames} Nikhil Byrapuram, Hwiseo (Irene) Choi, Adam Ge, Selena Ge, Tanya Khovanova, Sylvia Zia Lee, Evin Liang, Rajarshi Mandal, Aika Oki, Daniel Wu, and Michael Yang, Card Games Unveiled: Exploring the Underlying Linear Algebra. math.HO arXiv: 2306.09280, (2023)

\bibitem{STEPMSS}  Eric Chen, William Du, Tanmay Gupta, Tanya Khovanova, Alicia Li, Srikar Mallajosyula, Rohith Raghavan, Arkajyoti Sinha, Maya Smith, Matthew Qian, and Samuel Wang, The Classification of Magic SET Squares, Recreational Mathematics Magazine, v. 7, issue 13, (2020), pp. 71--94.

\bibitem{CragerEtAl} Julia Crager, Felicia Flores, Timothy E.~Goldberg, Lauren L.~Rose, Daniel Rose-Levine, Darrion Thornburgh, and Raphael Walker, How many cards should you lay out in a game of EvenQuads? A study of 2-caps in $AG(2, n)$. math.CO arXiv: 2212.05353, (2022).

\bibitem{DK} J.~D\'{e}nes and A.D.~Keedwell, Latin squares and their applications, New York-London: Academic Press, 1974.

\bibitem{MGGG} Liz McMahon, Gary Gordon, Hannah Gordon, and Rebecca Gordon. The Joy of SET. The Many Mathematical Dimensions of a Seemingly Simple Card Game. Princeton University Press, 2017.

\bibitem{Riordan} J.~Riordan, \textit{An Introduction to Combinatorial Analysis}, Wiley, 1958,

\bibitem{Rose} Lauren L.~Rose, Quads: A SET-like game with a Twist.

\bibitem{Vinci} Jim Vinci, The maximum number of sets for N cards and the total number of internal sets for all partitions of the deck, available at \url{https://www.setgame.com/sites/default/files/teacherscorner/SETPROOF.pdf}, accessed in 2022.

\bibitem{OEIS} OEIS Foundation Inc. (2023), The On-Line Encyclopedia of Integer Sequences, Published electronically at \url{https://oeis.org}.


\end{thebibliography}
\end{document}